\documentclass{amsart}
\usepackage{amsmath, amscd, amssymb, amsthm}
\usepackage{bbm}
\usepackage{latexsym}
\usepackage{amsfonts}
\usepackage{graphicx}
\usepackage[all,cmtip]{xy}
\usepackage[colorlinks,linkcolor=blue,breaklinks=blue,urlcolor=blue,citecolor=blue,anchorcolor=blue,pagebackref]{hyperref}%
\usepackage{geometry}
\newtheorem{theorem}{Theorem}
\newtheorem{lemma}{Lemma}

\newtheorem{remark}{Remark}
\newtheorem{proposition}{Proposition}

\newtheorem{conjecture}{Conjecture}

\renewcommand*\backref[1]{}
\renewcommand*\backrefalt[4]{ \ifcase #1 \or (cited on page #2) \else (cited on pages #2) \fi}

\newcommand{\be}{\begin{equation}}
\newcommand{\ee}{\end{equation}}
\newcommand{\bea}{\begin{eqnarray}}
\newcommand{\eea}{\end{eqnarray}}

\newcommand{\vs}{\vspace{0.5cm}}

\def\XXint#1#2#3{{\setbox0=\hbox{$#1{#2#3}{\int}$ }
\vcenter{\hbox{$#2#3$ }}\kern-.6\wd0}}

\begin{document}

\title[On Gauduchon K\"ahler-like manifolds]{On Gauduchon K\"ahler-like manifolds }
	
\author{Quanting Zhao}
\address{Quanting Zhao. School of Mathematics and Statistics \&
Hubei Key Laboratory of Mathematical Sciences, Central China Normal
University, Wuhan, 430079, P.R.China.} \email{zhaoquanting@126.com;zhaoquanting@mail.ccnu.edu.cn}
\thanks{Zhao is partially supported by National Natural Science Foundations of China with the grant No.11801205.
Zheng is partially supported by National Natural Science Foundations of China
with the grant No.12071050 and Chongqing Normal University.}

\author{Fangyang Zheng}
\address{Fangyang Zheng. School of Mathematical Sciences, Chongqing Normal University, Chongqing 401331, China}
\email{20190045@cqnu.edu.cn} \thanks{}


\begin{abstract}
In a paper by Angella, Otal, Ugarte, and Villacampa, the authors conjectured that on a compact Hermitian manifold, if a Gauduchon connection other than Chern or Strominger is K\"ahler-like, then the Hermitian metric must be K\"ahler. They also conjectured that if two Gauduchon connections are both K\"ahler-like, then the metric must be K\"ahler. In this paper, we discuss some partial answers to the first conjecture, and give a proof to the second conjecture. In the process, we discovered an interesting `duality' phenomenon amongst Gauduchon connections, which seems to be intimately tied to the question, though we do not know if there is any underlying reason for that from physics.
\end{abstract}

\maketitle

\tableofcontents

\markleft{Quanting Zhao and Fangyang Zheng}
\markright{Gauduchon K\"ahler-like}

\section{Introduction and statement of results}\label{intro}

The study of non-K\"ahler geometry has generated a lot of interests in recent years, partly due to the need from physics (see for instance \cite{GHR, IP, Strominger,  TsengYau}). As a sample, we refer the readers to the work \cite{AI, AU, Belgun, Belgun1, Gauduchon, Gauduchon1, Bismut}, \cite{EFV, FinoTomassini, FinoVezzoni, OV, Popovici}, \cite{Fu, FuYau, FuLiYau, LiYau}, \cite{LiuYang, LiuYang1, LiuYang2}, \cite{S, STW, T}, \cite{KYZ, VYZ, WYZ, Zheng}, and the references therein for more discussions. Here we concern ourselves with a specific topic: the rigidity problem for K\"ahler-like Gauduchon connections.

Given a Hermitian manifold $(M^n,g)$, there are three distinguished metric connections that are well-studied: the Riemannian (Levi-Civita) connection $\nabla$, the Chern connection $\nabla^c$, and the Strominger connection $\nabla^s$ (also known as the Bismut connection). When $g$ is K\"ahler, all three coincides, and when $g$ is not K\"ahler, they are mutually distinct.

The line joining $\nabla^c$ and $\nabla^s$ gives a one-parameter family of canonical Hermitian connections known as the {\em Gauduchon connections} \cite{Gauduchon1}, denoted by
\begin{equation}
D^{r} = \frac{1+r}{2}\nabla^c+   \frac{1-r}{2} \nabla^s, \ \ \ {r} \in {\mathbb R}.
\end{equation}
In this notation, $D^{1}=\nabla^c$, $D^{-1}=\nabla^s$, and $D^0=\frac{1}{2}(\nabla^c+\nabla^s)$ is the Hermitian projection of the Riemannian (Levi-Civita) connection $\nabla$, which in some literature is called the {\em Lichnerowicz connection}. Note that when $g$ is K\"ahler, all $D^r$ are equal (to the Riemannian connection), while when $g$ is not K\"ahler, all $D^{r}$ are distinct.

Recall that a metric connection $D$ (that is, $Dg=0$) on a Hermitian manifold $(M^n,g)$ is called {\em K\"ahler-like,} if its curvature tensor $R^D$ satisfies the symmetry conditions
\begin{gather}
R^D(x,y,z,w) + R^D(y,z,x,w) + R^D(z,x,y,w)=0, \label{bnc}\\
R^D(Jx,Jy,z,w) = R^D(x,y,z,w) = R^D(x,y,Jz,Jw), \label{type_1}
\end{gather}
for any real tangent vectors $x$, $y$, $z$, $w$ in $M^n$, where $J$ is the almost complex structure.

When the connection $D$ is {\em Hermitian,} meaning $Dg=0$ and $DJ=0$, it is shown in \cite[Remark 6]{AOUV} that $D$ is K\"ahler-like if and only if
$$ R^D(X,Y,Z, \overline{W}) = 0, \ \ \ R(X,\overline{Y},Z,\overline{W}) = R(Z,\overline{Y},X,\overline{W})$$
for any type $(1,0)$ tangent vectors $X$, $Y$, $Z$, $W$ in $M^n$, since $R^D(\ast , \ast , Z, W)=0$ always holds by the assumption $DJ=0$. So for Hermitian connections, the K\"ahler-like condition simply means that the only possibly non-zero components of the curvature tensor are $R^D(X,\overline{Y},Z,\overline{W})$, and $R^D$ is symmetric when the first and the third positions are interchanged.

Note that in an earlier version of our paper \cite{ZZ}, we made a mistake in stating an equivalent description of K\"ahler-likeness, which was kindly pointed out to us by Fino and Tardini. In \cite{FT} they examined the definition in more details and clarified the subtleties, constructed interesting new examples of Strominger K\"ahler-like manifolds and studied the preservation problem of this curvature condition under the pluriclosed flow.

This notion of K\"ahler-likeness was introduced in \cite{YZ} in 2016 for the Riemannian and Chern connections, following the pioneer work of Gray \cite{Gray} and others. Angella, Otal, Ugarte and Villacampa \cite{AOUV} generalized it to all metric connection, and they particularly studied it for the  one-parameter family of Gauduchon connection $D^r$. They proposed the following two conjectures:

\begin{conjecture}[AOUV\cite{AOUV}]\label{cj1}
For a compact Hermitian manifold $(M^n,g)$, if the Gauduchon connection $D^{r}$ is K\"ahler-like and ${r}\neq \pm 1$, then $g$ must be K\"ahler.
\end{conjecture}

\begin{conjecture}[AOUV\cite{AOUV}]\label{cj2}
For a compact Hermitian manifold $(M^n,g)$, if there are two real values $r\neq r'$ such that $D^{r}$ and $D^{r'}$ are both K\"ahler-like, then $g$ must be K\"ahler.
\end{conjecture}

Note that for each $n\geq 3$, there are examples of compact Hermitian manifold $(M^n,g)$ with $g$ non-K\"ahler such that its Riemannian (or Chern, or Strominger) connection is K\"ahler-like. In the Strominger connection case there are such examples of $n=2$ as well. Hence, the assumption ${r}\neq \pm 1$ in Conjecture \ref{cj1} is necessary.

For Conjecture \ref{cj1}, we observe that the following partial result is established:

\begin{theorem}\label{thm1}
Let $(M^n,g)$ be a compact Hermitian manifold of complex dimension $n$, whose Gauduchon connection $D^{r}$ is K\"ahler-like.
If one of the following occurs
\begin{enumerate}
\item $n=2$ and $r\neq -1$,
\item $n\geq 3, r\not\in ( -3-2\sqrt{3},0) \cup(0, \,  -3+2\sqrt{3})$ and $r\neq 1$,
\end{enumerate}
then $g$ is K\"ahler.
\end{theorem}

\begin{remark}
The second case (except when $r=0$) of Theorem \ref{thm1}, namely, $r\not\in ( -3-2\sqrt{3}, \,  -3+2\sqrt{3})$ and $r\neq 1$ implies the K\"ahlerness of $g$ when $n\geq3$, is due to Fu and Zhou \cite[Corollary 5.7]{Fu-Zhou}, where they proved the result by using properties on  total scalar curvatures.
\end{remark}

Both cases in Theorem \ref{thm1}, except $r=0$ in the second case, were proved in \cite{YZ1} under the stronger assumption that $D^r$ is flat. However, we note that the same proof works in the case of K\"ahler-like $D^r$ as well, since \cite[Lemma 3.1]{YZ1} still holds under the K\"ahler-like $D^r$ assumption. So the main novelty here is the $r=0$ case, namely, if a  compact Hermitian manifold has K\"ahler-like Lichnerowicz connection $D^0$, then the metric  must be K\"ahler.

Through a detailed case by case analysis, the authors of \cite{AOUV} verified the validity of Conjecture \ref{cj1} for all complex nilmanifolds and Calabi-Yau type solvmanifolds of complex dimension $n=3$.

The main purpose of this article is to confirm Conjecture \ref{cj2}. Actually, one could even drop the compactness assumption in this case.

\begin{theorem} \label{thm2}
Given a Hermitian manifold $(M^n,g)$, if there are two real values $r\neq r'$ such that $D^{r}$ and $D^{r'}$ are both K\"ahler-like,
then $g$ is K\"ahler.
\end{theorem}

In the proof of this theorem, we found an interesting ``duality" phenomenon amongst the Gauduchon connections $D^r$. It is given by the expression $\xi (r) = r/(2r-1)$. Note that it is a diffeomorphism from ${\mathbb R}\setminus \{ \frac{1}{2} \}$ onto itself satisfying $\xi (\xi (r)) =r$. Hence, for any $r\neq \frac{1}{2}$, one can consider $D^{\xi (r)}$ as the {\em dual} of $D^r$. Note that $D^{\frac{1}{2}}$ is special in the sense that its K\"ahler-likeness will imply that $g$ is K\"ahler. This was observed in \cite{YZ1} already when $D^{\frac{1}{2}}$ is assumed to be flat.

Note that the only fixed points of $\xi$ are $r=1$ and $r=0$, so the Chern connection $D^1$ and Lichnerowicz connection $D^0$ are ``self-dual" in some sense, while the other Gauduchon connections form duality pairs. In particular, the dual of Strominger conneciton $D^{-1}$ is $D^{\frac{1}{3}}$, which was called the {\em minimal connection} in \cite{YZ1} since it has the smallest torison norm amongst all Gauduchon connections. Also, for $r=-3+2\sqrt{3}$, $\xi (r)= -3-2\sqrt{3}$, so the two boundary points appeared in the second case of Theorem \ref{thm1} also form a duality pair.

While we do not know if this pairing phenomenon amongst the Gauduchon connections has any deeper implication in geometry or physics, we do observe here that, if $D^r$ and $D^{r'}$ are both K\"ahler-like, for two distinct real values $r$, $r'$ that are not a pair, namely, $r'\neq \xi (r)$, then one can easily show that $g$ must be K\"ahler, in a relatively straight forward way. Therefore, the main part of the proof of Theorem \ref{thm2} hinges on the case when $r'=\xi (r)$.

Next let us  consider the K\"ahler-like condition for the plane of connections spanned by the Gauduchon line and the Riemannian connection $\nabla$ on a Hermitian manifold $(M^n,g)$. When $g$ is not K\"ahler, $\nabla$ does not lie on the Gauduchon line. Hence $\nabla$ and $D^r$ together span a plane of canonical metric connections on $(M^n,g)$:
\begin{equation}
D^{r}_{s} = (1-s)D^r + s\nabla, \ \ \ r,s \in {\mathbb R}.
\end{equation}
We will call $D^{r}_{s}$ the {\em canonical $(r,s)$-connection.} Note that the parameter space for this plane of canonical connections is actually
$$\Omega = {\mathbb R}^2 \setminus L_1^{\ast}, \ \ \ L_1^{\ast} = \{ (r,1) \mid r\neq 0\} ,  $$
namely, the complement of the punctured horizontal line $L_1^{\ast}$ in the $rs$-plane, since $D^r_{1}=\nabla$ for any $r
\in \mathbb{R}$.

As $D^0_1=\nabla$ is the Riemannian connection, denote by $D^0_{\!-\!1}=\nabla'$ the mirror reflection of $\nabla$ with respect to the Gauduchon line
$L_0 = \{(r,s) \big| s=0\}$. We will call $\nabla'$ the {\em anti-Riemannian connection}. It turns out that $\nabla'$ is K\"ahler-like if and only if $\nabla$ is K\"ahler-like. Hence, for any $n\geq 3$, there are examples of compact Hermitian manifolds which are non-K\"ahler but have K\"ahler-like $\nabla$ and $\nabla'$. Another pair of points in $\Omega$ also turns out to be quite special, namely $(-1,2)$ and $(\frac{1}{3}, -2)$. Let us denote these two special connections by
$$ \nabla^+ = D^{\!-\!1}_2\quad \text{and} \quad \nabla^- = D^{\frac{1}{3}}_{\!-\!2}.$$
These two connections $\nabla^+$ and $\nabla^-$ can be expressed in terms of the Riemannian, Chern, and Strominger connections as
$$\nabla^+ =2\nabla - \nabla^s\quad \text{and} \quad\nabla^-=2\nabla^c - \nabla^+=2\nabla^c+\nabla^s -2 \nabla = 2\nabla'-\nabla^s.$$

A natural question is when will $D^r_s$ be K\"ahler-like? For the two special connections $\nabla^+$ and $\nabla^-$ above, by use of the main result in \cite{ZZ}, we will show that

\begin{theorem} \label{thm3}
For any Hermitian manifold $(M^n,g)$, $\nabla^+$ or $\nabla^-$ is K\"ahler-like if and only if the Strominger connection $\nabla^s$ is K\"ahler-like.
\end{theorem}

Note that for each $n\geq 2$, there are compact Hermitian manifolds that are Strominger K\"ahler-like but non-K\"ahler. Such manifolds are rather restrictive and interesting. In \cite{ZZ} and \cite{ZZ1}, we showed that such manifolds are always pluriclosed and classified them amongst all compact complex nilmanifolds endowed with nilpotent complex structures. Then a classification theorem \cite{YZZ} was also proved for all compact Strominger K\"ahler-like manifolds of dimension $3$. Let
$$\Omega'=\Omega \setminus \big( \{ s=0\} \cup \{ (0,1), (0,-1), (-1,2), (\frac{1}{3}, -2)\} \big) $$
 for the complement of Gauduchon line and the four special connections $\nabla$, $\nabla'$, $\nabla^+$, $\nabla^-$. Then we have the following:

\begin{theorem} \label{thm4}
Let $(M^n,g)$ be a Hermitian manifold such that $D^r_{s}$ is K\"ahler-like for some $(r,s)\in \Omega'$. Then  $g$ is K\"ahler.
\end{theorem}

The proof of this theorem will be divided into two cases: (i) we deal with the $rs\neq 0$ and $(r,s)\neq (-1,2), (\frac{1}{3}, -2)$ in Lemma \ref{lemma13} and Lemma \ref{lemma16}; (ii) we deal with the vertical line case: $r=0$ and $s\neq 0,1,-1$ in Lemma \ref{lemma15}.

One may also consider the generalized question of Conjecture \ref{cj2} for the connections in $\Omega$, namely, if the canonical metric connections $D^r_s$ and $D^{r'}_{s'}$ are both K\"ahler-like, for two distinct points $(r,s)$, $(r',s')$ $\in \Omega$, then must $g$ be K\"ahler?

Obviously, the cases should be ruled out when the two connections happen to be the pair $\{ \nabla , \nabla'\}$ or any pair out of the set $\{ \nabla^s, \nabla^+, \nabla^-\}$, as $g$ doesn't have to be K\"ahler at that time. It turns out that, for the other cases, the answer to the question above is yes, where the compactness assumption is not needed just like in Theorem \ref{thm2}. 

\begin{theorem} \label{thm5}
Let $(M^n,g)$ be a Hermitian manifold whose connections $D^r_{s}$ and $D^{r'}_{s'}$ are K\"ahler-like,
where $(r,s)$ and $(r',s')$ are distinct points in $\Omega$. If the two connections above are not the pairs:
$ \{ \nabla , \nabla'  \}$, $\{ \nabla^+ , \nabla^-  \}$, $\{ \nabla^+ , \nabla^s  \}$, $ \{ \nabla^- , \nabla^s  \}$,
then $g$ must be K\"ahler.
\end{theorem}

Note that the result above can be regarded as an extension of \cite[Theorem 2]{YZ}, which shows that a compact Hermitian manifold which is both Chern K\"ahler-like and Riemannian K\"ahler-like must be K\"ahler, to the noncompact case.

Theorem \ref{thm2} and Theorem \ref{thm5} motivate us to propose the following

\begin{conjecture}
Let $(M^n,g,g')$ be a compact Hermitian manifold endowed with two possibly different Hermitian metrics $g$ and $g'$.
Assume that the connection $D^r_s$ associated to $g$ and the one $D^{r'}_{s'}$ associated to $g'$ are both K\"ahler-like,
where $(r,s)$ and $(r',s')$ are distinct points in $\Omega$ such that the pair is not one of the four exceptional pairs listed in Theorem \ref{thm5}. Then $M^n$ admits a K\"ahler metric.
\end{conjecture}

Several partial cases have already been confirmed in light of \cite[Theorem 6]{ZZ}, which shows that non-K\"ahler compact Strominger K\"ahler-like manifolds admit no balanced metric. Hence the cases for $\{\nabla^s,\nabla^c\}$, $\{\nabla^s,\nabla\}$ and $\{\nabla^s,D^0\}$ are established, as compact Chern, Riemannian or Lichnerowicz K\"ahler-like manifold is necessarily balanced by \cite{YZ} and Proposition \ref{GKLr=0} below.

\vspace{0.3cm}

\section{Properties of Gauduchon K\"ahler-like manifolds}\label{gkl}

Let $(M^n,g)$ be a Hermitian manifold with $\omega$ the associated K\"ahler form of $g$. Denote by $\nabla$ the Riemannian (Levi-Civita) connection, and by
$$D^r=\frac{1+r}{2}\nabla^c + \frac{1-r}{2}\nabla^s, \ \ \ r\in {\mathbb R}$$
the Gauduchon connections, where  $\nabla^c$ is the Chern connection and $\nabla^s$ is the Strominger (also known as the Bismut) connection.

Fix any $p \in M^n$, let $\{ e_1, \ldots , e_n\} $ be a frame of $(1,0)$-tangent vectors of $M^n$ in a neighborhood of $p$, with $\{ \varphi_1, \ldots , \varphi_n\}$ being its dual coframe of $(1,0)$-forms. The symbols $e=\,^t\!(e_1, \ldots , e_n)$ and $\varphi = \,^t\!( \varphi_1, \ldots , \varphi_n)$ are reserved for the column vectors. Let  $\langle \ , \rangle $ be the (real) inner product given by the Hermitian metric $g$, and extend it bilinearly over ${\mathbb C}$. Following the notations of  \cite{YZ,YZ1}, we may write under the frame $e$:
\[\begin{cases}\nabla e = \theta_1 e + \overline{\theta_2 }\,\overline{e} \\
\nabla \overline{e} = \theta_2 e + \overline{\theta_1}\,\overline{e}
\end{cases}\]
with the matrices of connection and curvature of $\nabla $ given by
$$ \hat{\theta } = \left[ \begin{array}{ll} \theta_1 & \overline{\theta_2 } \\ \theta_2 & \overline{\theta_1 }  \end{array} \right] , \ \  \  \hat{\Theta } = \left[ \begin{array}{ll} \Theta_1 & \overline{\Theta}_2  \\ \Theta_2 & \overline{\Theta}_1   \end{array} \right], $$
 where
\begin{eqnarray}
\Theta_1 & = & d\theta_1 -\theta_1 \wedge \theta_1 -\overline{\theta_2} \wedge \theta_2, \nonumber \\
\Theta_2 & = & d\theta_2 - \theta_2 \wedge \theta_1 - \overline{\theta_1 } \wedge \theta_2,  \nonumber \\
d\varphi & = & - \ ^t\! \theta_1 \wedge \varphi - \ ^t\! \theta_2 \label{eq:structure12}
\wedge \overline{\varphi } .
\end{eqnarray}
Similarly, let $\theta$ and $\tau$ be respectively the connection matrix and torsion column vector under $e$ for the Chern connection $\nabla^c$, then the structure equations and Bianchi identities are
\begin{eqnarray}
d \varphi & = & - \ ^t\!\theta \wedge \varphi + \tau,  \label{eq:structure1} \\
d  \theta & = & \theta \wedge \theta + \Theta. \nonumber \\
d \tau & = & - \ ^t\!\theta \wedge \tau + \ ^t\!\Theta \wedge \varphi, \label{eq:Bianchi1}\\
d  \Theta & = & \theta \wedge \Theta - \Theta \wedge \theta. \nonumber
\end{eqnarray}
Note that the entries of the curvature matrix $\Theta$ are all $(1,1)$-forms, while the entries of the column vector $\tau $ are all $(2,0)$-forms, under any frame $e$. Let $\gamma = \theta_1-\theta$, with $\gamma = \gamma' + \gamma ''$ regarded as its decomposition into $(1,0)$ and $(0,1)$-parts, then $\gamma$ is the matrix under $e$ of the tensor $D^0-\nabla^c$, while $\overline{\theta_2}$ is the matrix under $e$ of the tensor $\nabla - D^0$.  It follows from \cite{YZ} that when $e$ is unitary, the matrices $\gamma$ and $\theta_2$ amount to
\begin{equation}\label{eq:gm+tht2}
\gamma_{ij} = \sum_k \{ T^j_{ik}\varphi_k  - \overline{T^i_{jk}} \,\overline{\varphi}_k \} , \ \ \ (\theta_2)_{ij} = \sum_k \overline{T^k_{ij} } \,\varphi_k ,
\end{equation}
where $T^k_{ij}$, satisfying $T^k_{ji}=-T^k_{ij}$, are the components of the Chern torsion, given by
\[\tau_k=\sum_{i,j=1}^n T^k_{ij} \varphi_i \varphi_j = 2 \sum_{i<j} T^k_{ij} \varphi_i \varphi_j .\]

As to the Gauduchon connection $D^r$, the matrices of connection and curvature under $e$ are
$$ \theta^r = \theta + (1-r) \gamma , \ \ \ \ \ \Theta^r = d\theta^r - \theta^r \wedge \theta^r.$$
It is clear from the definition that the Gauduchon connection $D^r$ is K\"ahler-like if and only if $\,^t\!\varphi \, \Theta^r=0$. Also, by the same proof of \cite[Lemma 4]{YZ}, it follows that

\begin{lemma}\label{lemma1}
Let $(M^n,g)$ be a Hermitian manifold. Given any $r \in {\mathbb R}$ and $p\in M$, there exists a unitary frame $e$ of $(1,0)$-tangent vectors in a neighborhood of $p$, such that the connection matrix $\theta^r|_p=0$.
\end{lemma}

In other words, one can always choose a local unitary frame such that the connection matrix of $D^r$ vanishes at a given point. Of course, the same property holds for any Hermitian connection $D$ on $M$, not just the Gauduchon connections $D^r$. Some $r \in {\mathbb R}$ and $p\in M$ will be frequently fixed in the calculation below, where a local unitary frame $e$ such that $\theta^r|_p=0$ is applied.

Modifying the first part of \cite[Lemma 3.1]{YZ1} for the Gauduchon flat case, we have the following:

\begin{lemma}\label{lemma2}
Let $(M^n,g)$ be a Hermitian manifold such that the Gauduchon connection $D^r$ is K\"ahler-like, where $r \neq 1$.
Then, under any local unitary frame $e$, the Chern torsion components satisfy
\begin{eqnarray}
& & T_{ik,j}^{\ell } \ = \    -(1+r ) \sum_{q}  T_{ik }^q T_{jq}^{\ell}   , \label{eq:rTderi} \\
& & r \sum_q ( T_{ij }^q T_{kq}^{\ell} + T_{ki }^q T_{jq}^{\ell} + T_{jk }^q T_{iq}^{\ell}  ) \ = \ 0, \label{eq:rTTcyclic}
\end{eqnarray}
for any $i$, $j$, $k$, $\ell$, where the indices after comma mean covariant derivatives with respect to $D^r$.
\end{lemma}

\begin{proof}
Fix any $p\in M$ and the identities above will be verified at $p$. As both sides are tensors, without loss of generality, the unitary frame $e$ with the vanishing $\theta^r$ at $p$ is applied. The unitary frame $e$ leads to the equality $\gamma'_{ij} = \sum_k T^j_{ik}\varphi_k$, which implies that
\begin{equation}
\, ^t\!\gamma' \varphi = - \tau . \label{eq:taugamma}
\end{equation}
The choice of $e$ above forces $\theta=(r -1) \gamma$ at $p$, which yields that
\begin{equation}  \label{eq:taugammavarphi}
\partial \varphi = -r  \, ^t\!\gamma' \varphi  = r\tau , \ \ \overline{\partial} \varphi =(r -1) \overline{\gamma'}\varphi, \ \ \Theta^r -\Theta = (1-r )d\gamma + (1-r)^2 \gamma \gamma.
\end{equation}
Since $D^r$ is K\"ahler-like, which is equivalent to $\,^t\!\varphi \, \Theta^{r}=0$, hence, in particular, the $(0,2)$-part of $\Theta^r$, and thus the $(2,0)$-part of $\Theta^r$, vanishes. It follows that,
$$ 0 = (\Theta^{r})^{2,0} = (1-r ) \partial \gamma' + (1-r )^2\gamma' \gamma'. $$
The assumption $r \neq 1$ enables us to get $\partial \gamma' + (1-r )\gamma' \gamma' =0$, which implies that
\begin{equation}
 T_{ik,j}^{\ell} - T_{ij,k}^{\ell} = 2r\, T^q_{kj}T^{\ell}_{iq} + (1-r) T^q_{ij}T^{\ell}_{kq} - (1-r) T^q_{ik}T^{\ell}_{jq}. \label{eq:middle}
 \end{equation}
By the equality (\ref{eq:taugamma}) and the fact that $\theta = (r-1)\gamma$ at $p$, the $(3,0)$-part of the first Bianchi identity $d\tau = \,^t\!\Theta \varphi - \,^t\!\theta \tau$ gives us
$$  \,^t\!\varphi \, \partial \gamma' - \partial \,^t\!\varphi \, \gamma ' = (r-1) \,^t\!\varphi  \,\gamma' \gamma'.$$
Hence, from $\partial \gamma' =- (1-r )\gamma' \gamma' $ above, it yields that $r \, \,^t\!\varphi  \,\gamma' \gamma' =0$, which is the identity (\ref{eq:rTTcyclic}) in the lemma.  After (\ref{eq:rTTcyclic}) is plugged into (\ref{eq:middle}), it follows that
\begin{equation}
 T_{ik,j}^{\ell} - T_{ij,k}^{\ell} =  (1+r) T^q_{ij}T^{\ell}_{kq} - (1+r) T^q_{ik}T^{\ell}_{jq}. \label{eq:middle2}
 \end{equation}
The index $(ijk)$ can be replaced by $(jki)$ and $(kij)$ in the equality above and all the three equalities are summed up, which yields
\begin{equation}
 -2( T_{ki,j}^{\ell} + T_{ij,k}^{\ell} + T_{jk,i}^{\ell} ) =  2(1+r) ( T^q_{ij} T^{\ell}_{kq} + T^q_{ki}T^{\ell}_{jq} + T^q_{jk} T^{\ell}_{iq}  ) . \label{eq:middle3}
 \end{equation}
By the comparison of (\ref{eq:middle3}) with (\ref{eq:middle2}), the identity (\ref{eq:rTderi}) of the lemma follows.
\end{proof}

Similarly, the modification of the second part of \cite[Lemma 3.1]{YZ1} for the Gauduchon flat case with one step further yields:

\begin{lemma}\label{lemma3}
Let $(M^n,g)$ be a Hermitian manifold such that the Gauduchon connection $D^r$ is K\"ahler-like. If $r=\frac{1}{2}$, then $g$ is K\"ahler. In general, under any local unitary frame $e$, the Chern torsion components satisfy
\begin{eqnarray}
4r (2r-1) T_{ik,\,\overline{\ell }}^j  & = & 4r^2(r-1) \sum_q T_{ik}^q \overline{ T_{j\ell }^q } + (r-1)(5r^2-1) \sum_q \{ T_{iq}^j \overline{ T_{\ell q}^k }  -   T_{kq}^j \overline{ T_{\ell q}^i }  \} \nonumber \\
 & & \ - (r-1)^3 \sum_q \{ T_{iq}^{\ell} \overline{ T_{j q}^k }  - T_{kq}^{\ell} \overline{ T_{j q}^i }  \}, \label{eq:rTderibar}
\end{eqnarray}
for any $i$, $j$, $k$, $\ell$, where the indices after comma mean covariant derivatives with respect to $D^r$.
\end{lemma}

\begin{proof}
Fix any $p\in M$ and let $e$ be a local unitary frame such that $\theta^r|_p=0$. From (\ref{eq:taugamma}), (\ref{eq:taugammavarphi}) and the $(2,1)$-part of the first Bichani identity $d\tau = \,^t\!\Theta \varphi - \,^t\!\theta \tau$, the $(2,1)$-part of $\, ^t\!\varphi \, \Theta^r =0$ amounts to
$$ \, ^t\!\varphi \wedge(r \overline{\partial} \gamma' - (r-1) \partial \,^t\!\overline{\gamma'} + r(r-1)\gamma' \,^t\!\overline{\gamma'} + r(r-1) \,^t\!\overline{\gamma'} \,\gamma' ) =0. $$
In terms of coefficients, it follows that $  P_{ik}^{j\ell }=0$, where
\begin{equation}
P_{ik}^{j\ell }  =  2r T^j_{ik,\overline{\ell}} + (r-1)y  - 2r(r-1) w - 2r(r-1) ( v^j_i - v^j_k) + (r-1)^2 (v^{\ell}_i - v^{\ell}_k ),  \label{eq:P=0}
\end{equation}
with the other symbols denoted by
\begin{gather*}
x = T^j_{ik,\overline{\ell}} - T^{\ell }_{ik, \overline{j}},
\quad y = \overline{ T^i_{j\ell , \overline{k}}  -  T^k_{j\ell , \overline{i}} },
\quad w  =  \sum_q T^q_{ik}\overline{T^q_{j\ell } },  \\
v^j_i = \sum_q T^j_{iq} \overline{T^k_{\ell q}},
\quad v^{\ell}_k = \sum_q T^{\ell}_{kq} \overline{T^i_{j q}},
\quad v^j_k = \sum_q T^j_{kq} \overline{T^i_{\ell q}} ,
\quad v^{\ell}_i = \sum_q T^{\ell}_{iq} \overline{T^k_{j q}}.
\end{gather*}
It yields, from $ P_{ik}^{j\ell} - P_{ik}^{\ell j} =0$, that
\begin{equation}
 2r x +2(r-1)y  =   4r(r-1)w + (r-1)(3r-1) \big( v^j_i-v^j_k - v^{\ell }_i + v^{\ell }_k \big). \label{eq:xy1}
\end{equation}
After $(ik)$ is interchanged with $(j\ell )$ in the identity \eqref{eq:xy1} above and complex conjugation is taken, it follows that
\begin{equation}
 2r y +2(r-1)x  =   4r(r-1)w + (r-1)(3r-1) \big( v^{\ell}_k-v^j_k - v^{\ell }_i + v^j_i \big). \label{eq:xy2}
\end{equation}
Denote by $2Q$ the right hand side of the two equalities \eqref{eq:xy1} and \eqref{eq:xy2} above, then it yields that
$$ rx+ (r-1)y = (r-1)x+ry = Q. $$
So if $r=\frac{1}{2}$, then it follows that $x=y$ and $Q=0$, namely, $ -w -\frac{1}{4}(v^j_i-v^j_k - v^{\ell }_i + v^{\ell }_k  ) =0$. Let $i=j$, $k=\ell$ and sum up $i$ and $k$. This implies $|T|^2 + |\eta |^2 =0$, where $\eta_k = \sum_i T^i_{ik}$, yielding that $T=0$, hence $g$ is K\"ahler.

On the other hand, if $r\neq \frac{1}{2}$, then the above system of linear equations \eqref{eq:xy1} and \eqref{eq:xy2} of $x$ and $y$ imply  $x=y=\frac{Q}{2r-1}$. After it is plugged back into (\ref{eq:P=0}), the identity stated in the lemma follows.
\end{proof}

Denote by $\eta_k$ the summation $\sum_i T^i_{ik}$ and thus $\eta = \sum_k \eta_k \varphi_k$ is the Gauduchon's torsion $1$-form. Let $i=j$ in the identity \eqref{eq:rTderibar} of Lemma \ref{lemma3} and sum it up, we obtain:

\begin{lemma} \label{lemma4}
Let $(M^n,g)$ be a Hermitian manifold with K\"ahler-like Gauduchon connection $D^r$. Then it holds that
\begin{equation*}
 4r(2r-1) \eta_{k,\overline{\ell}} = 4r^2(r-1)A_{k\overline{\ell}} +(r-1)(5r^2-1) (\overline{\phi^k_{\ell}} - A_{k\overline{\ell}} ) -(r-1)^3 ( B_{k\overline{\ell}}  - \phi_k^{\ell} )
\end{equation*}
for any indices $k$, $\ell$, where the indices after comma mean covariant derivatives with respect to $D^r$.
\end{lemma}

Here we adopted the notation in \cite{ZZ} that
$$ A_{k\overline{\ell}} = \sum_{i,j} T^i_{jk} \overline{ T^i_{j\ell }}\, , \ \ \ B_{k\overline{\ell}} = \sum_{i,j} T^{\ell}_{ij} \overline{ T^{k}_{ij}}\, , \ \ \  \phi_k^{\ell} = \sum_i T^{\ell}_{ki} \overline{\eta_i}\,. $$
Note that $\mbox{tr}(A) = \mbox{tr}(B) =|T|^2=\sum_{i,j,k} |T^i_{jk}|^2$, and $\mbox{tr}(\phi )=|\eta|^2 =\sum_k |\eta_k|^2$.

\begin{lemma} \label{lemma5}
Let $(M^n,g)$ be a Hermitian manifold with K\"ahler-like Gauduchon connection $D^r$. Then it holds:
\begin{equation*}
 2(2r-1) \chi = (r-1)(3r-1) |\eta|^2 - (r-1)^2 |T|^2 ,
\end{equation*}
where $\chi = \sum_k \eta_{k,\overline{k}}$ and the indices after comma mean covariant derivatives with respect to $D^r$.
\end{lemma}

\begin{proof}
Let $k=\ell$ and sum up in Lemma \ref{lemma4}. Then the identity in Lemma \ref{lemma5} with both sides multiplied by $r$ will follow. This establishes the identity in Lemma \ref{lemma5} when $r\neq 0$.

To cover the $r=0$ case, we note that if we let $i=j$ and $k=\ell$ and sum them up, then $x$ becomes $2\chi$, $y$ becomes $2\overline{\chi}$, and the quantity $2Q$, which is the common right hand side of (\ref{eq:xy1}) and (\ref{eq:xy2}), becomes
$$ 4r(r-1)|T|^2 + (r-1)(3r-1) (2|\eta|^2-2|T|^2) = 2(r-1)(3r-1)|\eta|^2 - 2(r-1)^2 |T|^2. $$
When $r\neq \frac{1}{2}$, the identities (\ref{eq:xy1}) and (\ref{eq:xy2}) lead to $x=y=\frac{Q}{2r-1}$ which proves Lemma \ref{lemma5}. When $r=\frac{1}{2}$, $g$ is K\"ahler so the equality holds as both sides are zero. This completes the proof of the lemma.
\end{proof}

The following proposition will be applied in the proofs of Theorem \ref{thm1} and Theorem \ref{thm5}.

\begin{proposition}\label{GKLr=0}
Let $(M^n,g)$ be a Hermitian manifold, where Lichnerowicz connection $D^0$ is K\"ahler-like. Then $\eta=0$ holds everywhere on $M^n$, namely, $g$ is balanced.
\end{proposition}

\begin{proof}
Lemma \ref{lemma3} for the case $r=0$ shows that
\begin{equation}\label{eq:RrTbarT:r=0}
\sum_q \{ T_{iq}^j \overline{ T_{\ell q}^k }  -   T_{kq}^j \overline{ T_{\ell q}^i }  \}
+  \sum_q \{ T_{iq}^{\ell} \overline{ T_{j q}^k }  - T_{kq}^{\ell} \overline{ T_{j q}^i }  \}=0
\end{equation}
for any $i$, $j$, $k$, $\ell$. Let $i=j$ in the equation above and sum $i$ up, which yields
\[\overline{\phi^{k}_{\ell}} - A_{k\bar{\ell}} + B_{k\bar{\ell}}-\phi^{\ell}_k=0\]
for any $k$, $\ell$. Since $A$ and $B$ are Hermitian symmetric matrices, it yields that
\begin{equation}\label{eq:ABphi}
A=B, \quad \phi = \phi^*.
\end{equation}

Assume that $|\eta|>0$ for some point $p$ on the manifold $M^n$. Then it is clear that the same also hold in a neighborhood of $p$, which enables us
to choose the unitary frame $e$ after some appropriate unitary transformation, such that $\frac{\eta}{|\eta|}=\varphi_n$ in such a neighborhood, yielding that $\eta_1 = \cdots =\eta_{n-1}=0$ and $\eta_n=|\eta| > 0$. Then $\phi_k^{\ell}=T^{\ell}_{kn}\overline{\eta_n}=T^{\ell}_{kn} |\eta| $, which implies that
\[\phi^{\ell}_n = \overline{\phi^n_{\ell}}=0\]
for any $\ell$. Since $\phi$ is Hermitian symmetric, after some another unitary transformation of $\varphi_1,\cdots,\varphi_{n-1}$ with $\varphi_n$ left unchanged, it follows that $\phi$ can be diagonal, yielding $T^\ell_{kn}=0$ for $k \neq \ell$ and $T^k_{kn}$ real for $k=\ell$.

Let $i=n$ in the identity \eqref{eq:RrTbarT:r=0}, which yields that
\begin{equation} \label{eq:Rrn1:r=0}
(T^{\ell}_{n \ell} - T^j_{n j})\overline{T^k_{j \ell}} = \sum_{q} T^j_{kq} \overline{T^n_{\ell q}} + T^{\ell}_{k q} \overline{T^n_{jq}}
\end{equation}
for any $j$, $k$, $\ell$. Similarly, let $j=n$ in the identity \eqref{eq:RrTbarT:r=0} and take the conjugation on both sides, which yields that
\begin{equation}\label{eq:Rrn2:r=0}
(T^k_{nk} + T^i_{n i})\overline{T^{\ell}_{ik}}=\sum_{q} T^i_{\ell q} \overline{T^n_{kq}} - T^k_{\ell q} \overline{T_{i q}^{n}},
\end{equation}
for any $i$, $k$, $\ell$. Hence, the two equalities \eqref{eq:Rrn1:r=0} and \eqref{eq:Rrn2:r=0} implies
\begin{equation}\label{eq:Rrkey:r=0}
T^{\ell}_{n \ell}\overline{T^k_{j \ell}} = \sum_{q} T^j_{k q} \overline{T^n_{\ell q}}
\end{equation}
for any $j$, $k$, $\ell$. Let $j=n$ in the identity \eqref{eq:Rrkey:r=0} and it follows that
\begin{equation}\label{orthogonal}
\sum_q T^n_{k q} \overline{T^n_{\ell q}} =
\begin{cases} (T^{\ell}_{n \ell})^2, &\ \text{if}\quad k=\ell \\ 0.&\ \text{if}\quad k \neq \ell\end{cases}
\end{equation}
Let $k=n$ in the identity \eqref{eq:Rrkey:r=0}, which yields that
\begin{equation}\label{eigen}
(T^\ell_{n \ell} + T^j_{n j})\overline{T^n_{j \ell}} =0
\end{equation}
for any $j,\ell$. These enables us to regard $\{T^n_{j \ell}\}$ as a $(n-1) \times (n-1)$ skew symmetric matrix, since $T^n_{n\ell}=T^n_{j n}=0$.
Without loss of generality, we can be assume that the first $m$ rows of the matrix $\{T^n_{j \ell}\}$ are non-zero ones and the remaining rows are zeros. It is clear that $\mathrm{rank}\{T^n_{j \ell}\}=m$, since the first $m$ rows are mutually orthogonal by \eqref{orthogonal}. It yields that the possibly non-zero elements of the matrix $\{T^n_{j \ell}\}$ lie in the intersection of the first $m$ rows and
the first $m$ columns. Then it follows that $m$ is a positive even integer, since \eqref{orthogonal} and $\sum_\ell T^\ell_{n \ell}=-|\eta|<0$ force $m \neq 0$, which yields $\det\{T^n_{j \ell}\} \neq 0$ for $1 \leq j, \ell \leq m$, while $\det\{T^n_{j \ell}\} = (-1)^m \det\{T^n_{j \ell}\} $ holds by the skew symmetry of $\{T^n_{j \ell}\}$. It is also clear that $T^\ell_{n \ell} \neq 0$ for $1 \leq \ell \leq  m$ by \eqref{orthogonal}. Then the analysis is narrowed to the matrix $\{T^n_{j \ell}\}$ with possibly smaller size, where $1 \leq j, \ell \leq m$.

Define an equivalent relation $\sim$ on the set $S=\{1,2,\cdots,m\}$ as follows:
\begin{enumerate}
\item For $1 \leq i \neq j \leq m$, the relation $i\sim j$ means that there exist finite mutually distinct indices $i_1,\cdots,i_k$ such that $i=i_1$, $j=i_k$ and $T^n_{i_1\,i_2}, T^n_{i_2\,i_3}, \cdots, T^n_{i_{k-1}\,i_{k}}$ are all non-zeros.
\item For any $i \in S$, the relation $i \sim i$ always holds.
\end{enumerate}
It is easy to check the '$\sim$' is indeed an equivalent relation on $S$. Then the equivalent class with a representative $i$ is denoted by $[i]$, namely $\{j\in S \big| j \sim i\}$. It follows clearly that $T^k_{nk}$ and $T^\ell_{n \ell}$ are real functions different by a sign when $T^n_{k \ell} \neq 0$ by \eqref{eigen} and thus the same holds when $k,\ell$ belong to one equivalent class $[i]$. Furthermore, $T^n_{k \ell}=0$ when $k\in [i]$, $\ell \in [j]$ and $[i]\neq [j]$, as $T^n_{k \ell}\neq0$ would imply $k \sim \ell$ and thus $[i]=[j]$, which is a contradiction.

Fix one equivalent class $[i]$. Then it follows from \eqref{orthogonal} and the discussion above that, for each $\ell \in [i]$,
\[ (T^\ell_{n \ell})^2 = \sum_{j=1}^m |T^n_{\ell j}|^2 = \sum_{j\in [i]}|T^n_{\ell j}|^2,\]
since $T^n_{\ell j}$ vanishes when $j$ goes out of $[i]$. Note that the number $(T^\ell_{n \ell})^2$ is a positive invariant when $\ell$ goes inside $[i]$, as it is shown above. From \eqref{eigen}, it yields that, for any $j$, $\ell$,
\[(T^\ell_{n \ell} + T^j_{nj})|T^n_{j \ell}|^2 = 0.\]
Sum up $j, \ell$ in the equivalent class $[i]$, which follows
\[\begin{aligned}
0 &= \sum_{j \in [i]}\sum_{\ell \in [i]} (T^\ell_{n \ell} + T^j_{nj})|T^n_{j \ell}|^2\\
  &= \sum_{j \in [i]}\sum_{\ell \in [i]} T^\ell_{n \ell}|T^n_{j \ell}|^2 + \sum_{j \in [i]}T^j_{nj}(T^j_{n j})^2\\
  &= \sum_{\ell \in [i]} T^\ell_{n \ell}(T^\ell_{n \ell})^2 + (\sum_{j \in [i]}T^j_{nj})(T^j_{n j})^2\\
  &= (\sum_{\ell \in [i]} T^\ell_{n \ell})(T^\ell_{n \ell})^2 + (\sum_{j \in [i]}T^j_{nj})(T^j_{n j})^2\\
  &= 2 (\sum_{\ell \in [i]} T^\ell_{n \ell})(T^\ell_{n \ell})^2.
\end{aligned}\]
It yields that $\sum_{\ell \in [i]} T^\ell_{n \ell}=0$. As the equivalent class $[i]$ can be arbitrary, it follows that
\[-|\eta|=\sum_{\ell =1}^m T^{\ell}_{n \ell} =0,\]
where a contradiction appears finally. Therefore, this completes the proof.
\end{proof}

Now we are ready to prove Theorem \ref{thm1}.

\begin{proof}[\textbf{Proof of Theorem \ref{thm1}}]
Note that the identity in Lemma \ref{lemma5} is exactly \cite[the formula (23)]{YZ1}. Hence, under the assumption that $D^r$ is K\"ahler-like, \cite[Lemma 3.3]{YZ1} still holds and it leads to \cite[the formula (25)]{YZ1}, which now takes the form
\begin{equation}
n(2r-1)\sqrt{-1} \partial \overline{\partial} \omega^{n-1} = \{ (r-1)^2 |T|^2 + (r^2+6r-3)|\eta|^2 \} \omega^n. \label{eq:Tetasquare}
\end{equation}
The compactness of $M^n$ enables us to integrate the above identity and the remaining part of argument in \cite{YZ1} gives a proof of the first case of Theorem \ref{thm1}, which also yields that $g$ is K\"ahler, for $n \geq 3$, when $r\not\in ( -3-2\sqrt{3}, \,  -3+2\sqrt{3})$ and $r\neq 1$. When $r=0$, Proposition \ref{GKLr=0} and \eqref{eq:Tetasquare} clearly imply that $g$ is K\"ahler. This completes the proof.
\end{proof}

For the sake of simplicity, let us denote by $T^j_{ik,\,\overline{\ell}}$ and $T^j_{ik | \overline{\ell}}$ the covariant derivatives with respect to Gauduchon connections $D^r$ and $D^{r'}$, respectively, where $r\neq r'$.  Since $D^{r'}-D^r = (r-r')\gamma$, it follows that
\begin{eqnarray}
T^j_{ik | \overline{\ell}} & = & \overline{e}_{\ell} T^j_{ik} - T^j_{qk} \langle D^{r'}_{\overline{e}_{\ell} } e_i, \overline{e}_q \rangle - T^j_{iq} \langle D^{r'}_{\overline{e}_{\ell} } e_k, \overline{e}_q \rangle  - T^q_{ik} \langle D^{r'}_{\overline{e}_{\ell} }  \overline{e}_j, e_q \rangle \nonumber \\
& = & T^j_{ik , \,\overline{\ell}} - (r-r') \{ T^j_{qk} \gamma_{iq}(\overline{e}_{\ell}) + T^j_{iq} \gamma_{kq}(\overline{e}_{\ell}) - T^q_{ik} \gamma_{qj}(\overline{e}_{\ell}) \} \nonumber \\
& = & T^j_{ik , \,\overline{\ell}} + (r-r') \{ T^j_{qk} \overline{T^i_{q\ell }}  + T^j_{iq} \overline{T^k_{q\ell }} - T^q_{ik} \overline{T^q_{j\ell }} \label{eq:deltaTbar} \}.
\end{eqnarray}
Let $i=j$ in the above identity and sum up, it yields
\begin{eqnarray}
&& \eta_{k|\overline{\ell}} = \eta_{k,\overline{\ell}} + (r'-r) \overline{\phi^k_{\ell} } \label{eq:deltaeta}.
\end{eqnarray}
Then let $k=\ell$ and sum up again, it follows
\begin{eqnarray}
&& \chi ' = \chi + (r'-r) |\eta|^2 \label{eq:deltachi}
\end{eqnarray}
where $\chi' = \sum_k \eta_{k|\overline{k}}$. As a consequence of Lemma \ref{lemma5} and (\ref{eq:deltachi}), it proceeds to the following
\begin{lemma} \label{lemma6}
Suppose that a Hermitian manifold $(M^n,g)$ has K\"ahler-like Gauduchon connections $D^r$ and $D^{r'}$, where $r\neq r'$. Then it holds that
$$ (2rr'-r-r') \{ |\eta|^2 + |T|^2 \} = 0. $$
\end{lemma}

\begin{proof}
Let us assume that $r$, $r'\neq \frac{1}{2}$, as otherwise $g$ will be K\"ahler, hence $T=\eta=0$. By Lemma \ref{lemma5}, it follows that $\chi = f(r) |\eta|^2 - h(r) |T|^2$ and $\chi ' = f(r') |\eta|^2 - h(r') |T|^2$, where
$$ f(r) = \frac{(r-1)(3r-1)}{2(2r-1)}, \ \ \ \ \ \ h(r) = \frac{(r-1)^2 }{2(2r-1)}. $$
Hence, it yields from (\ref{eq:deltachi}) that
\begin{eqnarray*}
(r'-r)|\eta|^2 & = & \chi'-\chi \ = \ (f(r')-f(r)) |\eta|^2 - (h(r')-h(r)) |T|^2 \\
& = & \frac{(r'-r)}{2(2r-1)(2r'-1)} \{ (6rr' -3r-3r'+2) |\eta|^2 - (2rr'-r-r') |T|^2 \}.
\end{eqnarray*}
Cancel the factor $r'-r\neq 0$, which yields the identity stated in the lemma.
\end{proof}

As a consequence, it is clear that when $r$ and $r'$ satisfy $2rr'-r-r'\neq 0$, $D^r$ and $D^{r'}$ cannot be simultaneously K\"ahler-like, unless $g$ is K\"ahler. Therefore the key in proving Theorem \ref{thm2} is to deal with the case when $r$ and $r'$ does satisfy the condition $2rr'-r-r'=0$.

Lemma \ref{lemma6} also suggests the following phenomenon: Gauduchon connections seem to form {\em duality pairs} in a sense, which will be described below. Consider the function from ${\mathbb R}\setminus \{ \frac{1}{2}\}$ onto itself
 $$\xi (r)=\frac{r}{2r-1}.$$
The graph of this function is the two branches of hyperbola centered at the point $(\frac{1}{2}, \frac{1}{2})$. The connection $D^{\xi (r)}$ can be considered as the {\em dual} of $D^r$, since $\xi (\xi (r))=r$, for any $r \in {\mathbb R}\setminus \{ \frac{1}{2}\}$, while the value $\frac{1}{2}$ can be regarded as the dual to $\infty $. Note that when $D^{\frac{1}{2}}$ is K\"ahler-like, the metric is necessarily K\"ahler. In this sense, the Lichnerowicz connection $D^0$ and the Chern connection $D^1$ are the only {\em self-dual} Gauduchon connections, while the Strominger (or Bismut) connection $D^{-1}$ is dual to $D^{\frac{1}{3}}$, the so-called {\em minimal} connection which has the smallest torsion norm amongst all Gauduchon connections. Finally, the two boundary values $-3\pm 2\sqrt{3}$ of the interval appeared in Theorem \ref{thm1} are also dual to each other. Therefore, the key to prove Theorem \ref{thm2} is to deal with the case when $r$ and $r'$ form a duality pair.

In the remaining part of this section, we will assume that $(M^n,g)$ is a Hermitian manifold, whose Gauduchon connections $D^r$ and $D^{r'}$ are both K\"ahler-like, where $r\neq \frac{1}{2}$, $r\neq r'$ and $r'=\xi (r) = \frac{r}{2r-1}$. This implies that both  $r$ and $r'$ do not belong to $\{ 0, 1, \frac{1}{2}\}$. Our goal is to conclude that $g$ must be K\"ahler, thus completing the proof of Theorem \ref{thm2}.

Similar to (\ref{eq:deltaTbar}), it yields that
\begin{eqnarray}
T^j_{ik | \ell} & = & e_{\ell} T^j_{ik} - T^j_{qk} \langle D^{r'}_{e_{\ell} } e_i, \overline{e}_q \rangle - T^j_{iq} \langle D^{r'}_{e_{\ell} } e_k, \overline{e}_q \rangle  - T^q_{ik} \langle D^{r'}_{e_{\ell} }  \overline{e}_j, e_q \rangle \nonumber \\
& = & T^j_{ik , \,{\ell}} - (r-r') \{ T^j_{qk} \gamma_{iq}(e_{\ell}) + T^j_{iq} \gamma_{kq}(e_{\ell}) - T^q_{ik} \gamma_{qj}(e_{\ell}) \} \nonumber \\
& = & T^j_{ik , \,\ell} - (r-r') \{ T^j_{qk} T^q_{i\ell }  + T^j_{iq} T^q_{k\ell } - T^q_{ik} T^j_{q\ell } \} \nonumber \\
& = & T^j_{ik , \,\ell}  + (r'-r) \{ T^q_{\ell i} T^j_{kq}   + T^q_{k\ell } T^j_{iq}  + T^q_{ik} T^j_{\ell q }\}. \label{eq:deltaT}
\end{eqnarray}
where the index after comma or `$|$' stands for covariant derivatives with respect to $D^r$ or $D^{r'}$ respectively as before.
Together with Lemma \ref{lemma2} and (\ref{eq:deltaT}), it follows that

\begin{lemma} \label {lemma7}
If a Hermitian manifold $(M^n,g)$ has K\"ahler-like Gauduchon connections $D^r$ and $D^{r'}$, where $r\neq 0,1$, $r'\neq 1$ and $r\neq r'$, then
$$ T^j_{ik,\, \ell } = T^j_{ik|\ell } = \sum_q T^q_{ik} T^j_{q\ell } =0 $$
for any indices $i$, $j$, $k$, $\ell$. In particular, $\eta_{i,j}=\eta_{i\,| j}=0$ and $C_{ij}:= \sum_{q,s} T^q_{si} T^s_{qj} =0$.
\end{lemma}

\begin{proof}
When $r\neq 0,1$, $ T^j_{ik,\, \ell } = T^j_{ik|\ell }$ follows from (\ref{eq:rTTcyclic}) in Lemma \ref{lemma2} and (\ref{eq:deltaT}). Since $r, r'\neq 1$ and $r\neq r'$, the identity (\ref{eq:rTderi}) in Lemma \ref{lemma2} for both derivatives imply \[\sum_q T^q_{ik} T^j_{q\ell } =0.\]
Hence both derivatives are also zeros.
\end{proof}

Now let us assume that $r\neq \frac{1}{2}$ and $r'=\xi(r)=\frac{r}{2r-1}\neq r$, which follows that neither $r$ nor $r'$ belongs to $\{ 0,1,\frac{1}{2}\}$. By Lemma \ref{lemma4}, it yields that
$$ \eta_{k,\,\overline{\ell}} = a(r) A + b(r) \,( \phi^{\ast} - A) +c(r)\, (  \phi - B ), $$
where $\phi = \phi_k^{\ell}$, $\phi^{\ast } = \overline{\phi^k_{\ell}}$, $A=A_{k\overline{\ell}}$, $B=B_{k\overline{\ell}}$ and
$$ a(r) = \frac{r(r-1)}{2r-1}, \ \ \ b(r) = \frac{(r-1)(5r^2-1)}{4r(2r-1)}, \ \ \ c(r) = \frac{(r-1)^3}{4r(2r-1)}.$$
The condition that $r'=\xi(r)$, or equivalently $2rr' = r+r'$, implies that
$$ a(r')-a(r) = (r'-r), \ \ \ b(r')-b(r) = \frac{3rr'+1}{4rr'}(r'-r), \ \ \ c(r')-c(r) = \frac{rr'-1}{4rr'} (r'-r). $$
After these are plugged into (\ref{eq:deltaeta}), namely, $\eta_{k|\overline{\ell}} - \eta_{k,\overline{\ell}} = (r'-r)\phi^{\ast }$, and the non-zero factor $(r'-r)$ is cancelled out, it follows that
$$ 4rr' \phi^{\ast } = 4rr' A + (3rr'+1) (\phi^{\ast }-A) + (rr'-1) (\phi -B), $$
or equivalently,
$$ (rr'-1) \{ A-B + \phi - \phi^{\ast} \} =0. $$
Since $rr'-1= \frac{r^2}{2r-1} -1 = \frac{(r-1)^2}{2r-1} \neq 0$, it yields that $A-B= \phi^{\ast }-\phi$. Note that the left hand side of the last equality is Hermitian symmetric, while the right hand is skew-Hermitian, which implies that both sides are zero, namely

\begin{lemma} \label {lemma8}
If a Hermitian manifold $(M^n,g)$ has K\"ahler-like Gauduchon connections $D^r$ and $D^{r'}$,
where $r\neq \frac{1}{2}$ and $r'=\frac{r}{2r-1}\neq r$,  then $A=B$, $\phi =\phi^{\ast }$ and
$$ \eta_{k,\overline{\ell }} = \frac{(r-1)(3r-1)}{2(2r-1)} \phi - \frac{(r-1)^2}{2(2r-1)} A. $$
\end{lemma}

Now we are ready to prove  Theorem \ref{thm2}.

\begin{proof}[\textbf{Proof of Theorem \ref{thm2}}]
Let $(M^n,g)$ be a Hermitian manifold, with K\"ahler-like Gauduchon connections $D^r$ and $D^{r'}$ where $r\neq r'$. It can be assumed that $r,r'\neq \frac{1}{2}$ and $2rr'-r-r'= 0$, as otherwise $g$ is K\"ahler by the first part of Lemma \ref{lemma3} and $2rr'-r-r'\neq 0$ will imply, by Lemma \ref{lemma6}, that the metric has vanishing Chern torsion hence is K\"ahler.

Let us start with $r'=\frac{r}{2r-1}\neq r$ and thus $r,r' \notin \{0,1,\frac{1}{2}\}$. By Lemma \ref{lemma8}, it yields that $\phi = \phi^{\ast}$. Also, Lemma \ref{lemma7} implies that $\sum_q T^q_{ik} T^j_{q\ell }=0$ for any indices $i$, $j$, $k$, and $\ell$. Multiply by $\overline{\eta}_k\overline{\eta}_{\ell}$ and sum up $k$ and $\ell$, which yields that
$$ \sum_q \phi_i^q \phi_q^j = 0$$
for any $i$, $j$. Since $\phi =\phi^{\ast}$, that is, $\phi_q^j = \overline{\phi^q_j}$, it follows that
$$ \sum_q \phi_i^q \overline{\phi_j^q} =0 $$
for any $i$, $j$, which implies $\phi=0$. The trace of $\phi$ is $|\eta|^2$ and thus $\eta =0$. From the last identity in Lemma \ref{lemma8}, it yields that $A=0$. Therefore $\mbox{tr}(A)=|T|^2=0$, implying that $T=0$ and $g$ is K\"ahler. This completes the proof of Theorem \ref{thm2}.
\end{proof}

\vspace{0.3cm}

\section{K\"ahler-like canonical $(r,s)$-connections}\label{canconn}

Let us turn our attention to the plane of canonical metric connections generated by the Gauduchon line $D^r$ and the Riemannian connection $\nabla$ on a Hermitian manifold $(M^n,g)$:
$$D^r_s=(1-s)D^r+s\nabla, \ \ \ (r,s)\in \Omega\subseteq {\mathbb R}^2, $$
where $ \Omega = \{ (r,s) | s\neq 1\} \cup \{ (0,1)\}$.  The points $(0,1)$, $(0,-1)$, $(1,0)$, $(-1,0)$ in $\Omega$ corresponds to the Riemannian, anti-Riemannian,  Chern, and Strominger connection, respectively. Each can be K\"ahler-like yet non-K\"ahler. Also, the two points $(-1,2)$ and $(\frac{1}{3},-2)$  turn out to be special as well, and the corresponding connections are denoted by $\nabla^+$, $\nabla^-$, respectively. Denote by
$$\Omega' = \Omega \setminus (\{s=0\} \cup \{ (0,1), (0,-1), (-1,2), (\frac{1}{3},-2)  \})$$
the complement of the Gauduchon line $L_0=\{(r,s) \big| s=0\}$ and the four special points corresponding to the Riemannian connection $\nabla$, the anti-Riemannian connection $\nabla'$, and $\nabla^+$, $\nabla^-$. Our goal is to show that, for any $(r,s)\in \Omega'$,  the canonical metric connection $D^r_s$ cannot be K\"ahler-like unless the metric is K\"ahler, while $\nabla^+$ or $\nabla^-$ being K\"ahler-like is equivalent to that the Strominger connection $\nabla^s$ is K\"ahler-like.

Let us try to understand the implication of a canonical $(r,s)$-connection $D^r_s$ being K\"ahler-like. Let $(M^n ,g)$ be a Hermitian manifold with $D^r_s$ being K\"ahler-like, where  $(r,s)\in \Omega$. What this means will be investigated in terms of the torsion components of the Chern connection of $g$. Fix a point $p\in M$ and let $e$ be a local unitary frame near $p$, with the dual coframe $\varphi$. Denote by $\theta$, $\tau$ the connection matrix and torsion vector of the Chern connection $\nabla^c=D^1$ under $e$. It follows, from the discussion before Lemma \ref{lemma1} in Section \ref{gkl}, that
$$ D^r_s e = (\theta + t\gamma )e + s\overline{\theta}_2 \overline{e}, $$
where $t= 1-r+rs$. That is, the matrices of connection and curvature of $D^r_s$ under the frame $\{ e, \overline{e}\}$ are:
$$ \theta^{\!D} = \left[ \begin{array}{cc} \theta^{(t)} & s\overline{\theta}_2 \\ s \theta_2 & \overline{\theta^{(t)} } \end{array} \right] , \ \ \ \Theta^{\!D} = \left[ \begin{array}{cc} \Theta^{\!D}_1 & \overline{\Theta^{\!D}_2} \\  \Theta^{\!D}_2 & \overline{\Theta^{\!D}_1} \end{array} \right] , $$
where $D=D^r_s$ and
$$ \Theta^{\!D}_1=d\theta^{(t)} - \theta^{(t)} \theta^{(t)} - s^2 \overline{\theta}_2 \theta_2, \ \ \ \ \ \Theta^{\!D}_2 = s(d\theta_2 - \theta_2\,\theta^{(t)} -  \overline{\theta^{(t)}}\, \theta_2), $$
while $\theta^{(t)}=\theta +t\gamma $ corresponds to the Gauduchon connection $D^{1-t} $. Its curvature tensor $R^D$ is given by
\begin{eqnarray*}
R^D_{XY\overline{i}\overline{j}} & = & (\Theta^{\!D}_2)_{ij}(X,Y)  \\
R^D_{XYi\overline{j}} & = & (\Theta^{\!D}_1)_{ij}(X,Y)
\end{eqnarray*}
for any tangent vector $X$, $Y$. By definition, $D$ being K\"ahler-like means $ R^D_{\overline{i}\overline{j}XY} =  R^D_{XY\overline{i}\overline{j}} = 0$ and $  R^D_{i\overline{j}k\overline{\ell}}=R^D_{k\overline{j}i\overline{\ell}}$ for any indices and any tangent vectors $X$, $Y$. Therefore, it yields that

\begin{lemma} \label{lemma9}
On a Hermitian manifold $(M^n,g)$, the canonical metric connection $D=D^r_s$ is K\"ahler-like if and only if
\begin{equation*}
\Theta^{\!D}_2 = 0, \ \ \ \ \  (\Theta^{\!D}_1)^{2,0}=0, \ \ \ \ \  \,^t\!\varphi \, (\Theta^{\!D}_1)^{1,1} =0. \label{eq:DKL}
\end{equation*}
\end{lemma}

As an immediate corollary, we observe the following duality phenomenon for the K\"ahler-likeness of canonical metric connections, which
should not be confused with the duality pairs on the Gauduchon line discussed in Section \ref{gkl}. It occurs in the subset $\Omega_0 = \{(r,s) \big| s\neq 0, 1, -1\} \cap \Omega$. Define a map $\Psi : \Omega_0 \rightarrow \Omega_0$ by
$$ \Psi (r,s) = (\frac{1-s}{1+s}r, -s). $$
Clearly, $\Psi (r,s) \neq (r,s)$ and $\Psi (\Psi (r,s)) = (r,s)$. For $(r',s') =  \Psi (r,s) $, it is easy to see that
$$ t'=1-r'+r's' = 1 - (1+s)r' = 1-(1-s)r=t. $$
Hence, for $D=D^r_s$ and $D'=D^{r'}_{s'}$, it yields that $\Theta^{D'}_1=\Theta^{D}_1 $ and $\Theta^{D'}_2=- \Theta^{D'}_2 $. Then, by Lemma \ref{lemma9}, it is clear that $D'$ will be K\"ahler-like if and only if $D$ is K\"ahler-like.

For instance, $\Psi (-1,2)= (\frac{1}{3}, -2)$, hence $\nabla^+$ is K\"ahler-like if and only if $\nabla^-$ is K\"ahler-like. It is obvious that the K\"ahler-likeness of $\nabla$ is equivalent to that of $\nabla'$, since $(0,-1)$ can also be regarded as the dual to $(0,1)$.

Note that this pairing phenomenon was caused by the reflection of the $s$-factor in the parameter plane. For any point on the punctured line $L_{-1}^{\ast }=\{ (r, -1) \mid r\neq 0\}$, there is no obvious candidate of another point in $\Omega$ such that the corresponding pair of connections will be K\"ahler-like simultaneously. The same thing goes for points on the Gauduchon line $L_0=\{ (r,s) \big| s=0\}$, where $\Psi $ is defined but has only fixed points.

For any given point $p$, let us choose our unitary frame $e$ such that $\theta^{(t)}|_p=0$. Then at the point $p$, it follows
from $\theta = -t\gamma$ and \eqref{eq:taugamma} that
$d\,^t\!\varphi = (1-t)\,^t\!\varphi\gamma' + t \,^t\!\varphi \overline{\,^t\!\gamma'}$, which implies
\[ \partial \varphi_q = (1-t) \sum_{i,k}T^q_{ik}\varphi_i\varphi_k,
\quad \overline{\partial} \varphi_q = t\sum_{i,j} \overline{T^i_{qj}} \, \varphi_i \overline{\varphi}_j, \]
and, by the first Bianchi identity $d\tau = \,^t\!\Theta \varphi - \,^t\!\theta \tau$,
\[\begin{split}
\,^t\!\varphi \,\Theta &= d\,^t\!\tau + \,^t\!\tau \theta = d(\,^t\!\varphi \gamma') + \,^t\!\varphi \gamma' (-t\gamma) \\
&= \,^t\!\varphi \{ (1-2t)\gamma' \gamma' - \partial \gamma' \} + \,^t\!\varphi \{ -\overline{\partial} \gamma' + t \gamma ' \overline{\,^t\!\gamma'} + t \overline{\,^t\!\gamma'} \gamma'\} \\
&= \,^t\!\varphi \{ -\overline{\partial} \gamma' + t \gamma ' \overline{\,^t\!\gamma'} + t \overline{\,^t\!\gamma'} \gamma'\} ,
\end{split}\]
where the last equality results from the vanishing of the $(2,0)$ part $\Theta^{2,0}$ of the Chern curvature and thus
\begin{equation}
 0= (^t\!\varphi \,\Theta )^{3,0} = \,^t\!\varphi \{ (1-2t)\gamma' \gamma' - \partial \gamma' \}. \label{eq:30part}
 \end{equation}
Since the entries of $\theta_2$ are $(1,0)$-forms by \eqref{eq:gm+tht2}, it follows that
\begin{eqnarray*}
(\Theta^{\!D}_1)^{2,0} & = & t \partial \gamma' + t^2 \gamma' \gamma', \\
(\Theta^{\!D}_1)^{1,1} & = & \Theta + t( \overline{\partial } \gamma' - \partial \overline{\,^t\!\gamma'}) - t^2 (\gamma' \overline{\,^t\!\gamma'} + \overline{\,^t\!\gamma'} \gamma' ) - s^2 \overline{\theta_2} \theta_2.
 \end{eqnarray*}
Therefore, the equations in Lemma \ref{lemma9} are equivalent to
\begin{eqnarray}
&&  s\,d\,\theta_2 = 0, \label{eq:dtheta2}\\
& &  t(\partial \gamma' + t \gamma' \gamma' ) =0, \label{eq:tgamma} \\
&& ^t\!\varphi \, \{ (t-1) \overline{\partial} \gamma' - t \,\partial \,\overline{ ^t\!\gamma'} - t(t-1) ( \gamma'\overline{ ^t\!\gamma'} + \overline{ ^t\!\gamma'} \gamma' ) - s^2\, \overline{\theta}_2 \theta_2 \} =0. \label{eq:main}
\end{eqnarray}
Note that the identity \eqref{eq:gm+tht2} also implies that
\begin{eqnarray*}
\overline{\partial} \,\overline{(\theta_2)}_{ik} & = &  -\sum_{j,\ell } \big( T^j_{ik,\overline{\ell}} + (t-1) \sum_q T^q_{ik} \overline{T^q_{j\ell }} \big) \overline{\varphi}_j \overline{\varphi}_{\ell}, \\
\partial \,\overline{(\theta_2)}_{ik} & = & \sum_{j,\ell } \big( - T^j_{ik,\ell} + t \sum_q T^q_{ik} T^j_{q\ell } \big) \overline{\varphi}_j \varphi_{\ell}, \\
\end{eqnarray*}
which enable us to express the equations (\ref{eq:30part}), (\ref{eq:dtheta2}), (\ref{eq:tgamma}) and (\ref{eq:main}) in terms of their components.
Hence, it yields

\begin{lemma} \label{lemma10}
Suppose that the canonical metric connection $D^r_s$ of a Hermitian manifold $(M^n,g)$ is K\"ahler-like. Then the Chern torsion components satisfy
\begin{eqnarray}
&&  {\mathfrak S}_{i,k,\ell} \{ T^j_{ik,\ell} + (3t-2) T^q_{ik} T^j_{\ell q} \} \,=\,0, \label{eq:27}\\
&& s \, \{ T^j_{ik,\ell} + t T^q_{ik} T^j_{\ell q} \} \ = \ 0, \label{eq:28}\\
&& s \,\{ T^j_{ik, \overline{\ell}} - T^{\ell}_{ik, \overline{j}} + 2(t-1) T^q_{ik}\overline{T^q_{j\ell } }    \} \, = \, 0, \label{eq:29} \\
& &  t\, \{ T^j_{ik,\ell} - T^j_{i\ell , k}  + 2(t-1) T^q_{k\ell }T^j_{iq} +t T^q_{ik }T^j_{\ell q} + t T^q_{\ell i}  T^j_{kq} \} \,=\,0, \label{eq:30} \\
&& 2(t-1)T^j_{ik,\overline{\ell}} + t (\overline{ T^i_{j\ell , \overline{k}} }- \overline{ T^k_{j\ell , \overline{i}}} ) \, = \, -2t(t-1) (w + v^j_i - v^j_k) + (t^2-s^2) (v_i^{\ell} - v_k^{\ell} ), \label{eq:31}
\end{eqnarray}
for any $i,j,k,\ell$, where the index after comma stands for covariant derivative with respect to $D^{1-t}$, $t=1-r+rs$, ${\mathfrak S}$ stands for the cyclic sum, while $w=\sum_q T^q_{ik} \overline{T^q_{j\ell}}$ and
\begin{gather*}
v^j_i = \sum_q T^j_{iq} \overline{T^k_{\ell q}},
\quad v^{\ell}_k = \sum_q T^{\ell}_{kq} \overline{T^i_{j q}},
\quad v^j_k = \sum_q T^j_{kq} \overline{T^i_{\ell q}} ,
\quad v^{\ell}_i = \sum_q T^{\ell}_{iq} \overline{T^k_{j q}},
\end{gather*}
are defined as in the proof of Lemma \ref{lemma3}.
\end{lemma}

Let us try to get the expression for $T^j_{ik,\overline{\ell}}$. As in the proof of Lemma \ref{lemma3}, denote by
\begin{eqnarray}
x & = & T^j_{ik,\overline{\ell}} - T^{\ell}_{ik,\overline{j}}\,, \ \ \ \ \  y \ = \  \overline{ T^i_{j\ell , \overline{k} } -  T^k_{j\ell , \overline{i}}}\,, \nonumber \\
Q  & = & -2t(t-1)w - \frac{1}{2}(3t^2-2t-s^2)(v_i^j - v_k^j - v_i^{\ell} + v_k^{\ell} )\,\nonumber.
\end{eqnarray}
Interchange $j$ and $\ell$ in the last identity of Lemma \ref{lemma10} and take the difference between the two, it yields
$$ (t-1)x + t y = Q.$$
Interchange $(ik)$ with $(j\ell)$ and take the complex conjugation in the equation above, we get
$$ (t-1)y + tx = Q.$$
When $t\neq \frac{1}{2}$, it implies that $x=y=\frac{Q}{(2t-1)}$. Plug this back into the last identity of Lemma \ref{lemma10},
it yields

\begin{lemma} \label{lemma11}
Suppose that the canonical metric connection $D^r_s$ of a Hermitian manifold $(M^n,g)$ is K\"ahler-like. Then the Chern torsion components satisfy the equation
$$ 4(t-1)(2t-1)T^j_{ik,\overline{\ell}}  = -4t(t-1)^2w -t(5t^2-10t+4+s^2) (v_i^j - v_k^j) + (t^3-3s^2t+2s^2) (v_i^{\ell} - v_k^{\ell}) $$
for any $i,j,k,\ell$, where the index after comma stands for covariant derivative with respect to $D^{1-t}$ and $t=1-r+rs$.
\end{lemma}

\begin{proof}
The case when $t \neq \frac{1}{2}$ has been shown. When $t=\frac{1}{2}$, the linear system above about $x$ and $y$ yields $x=y$ and $Q=0$, that is
\begin{equation}
 w + (\frac{1}{4} + s^2) ( v_i^j - v_k^j - v_i^{\ell} + v_k^{\ell} ) = 0. \label{eq:thalf}
 \end{equation}
Therefore the identity in Lemma \ref{lemma11} also holds.
\end{proof}

\begin{lemma} \label{lemma12}
Suppose that the canonical metric connection $D^r_s$ of a Hermitian manifold $(M^n,g)$ is K\"ahler-like. Then it yields that,
when $t \neq \frac{1}{2}$, it holds that, for any $i$, $k$,
\begin{equation}
4s(t-1)^2 \sum_q |T^q_{ik}|^2= s (3t^2-2t -s^2) \, \sum_q \{ 2Re (T^i_{iq} \overline{ T^k_{kq}}) - |T^i_{kq}|^2 - |T^k_{iq}|^2 \},  \label{eq:key}
\end{equation}
when $t=\frac{1}{2}$, the identity above holds without the $s$ factor, namely,
\begin{equation}
\sum_q |T^q_{ik}|^2= (\frac{1}{4}+s^2) \, \sum_q \{ 2Re (T^i_{iq} \overline{ T^k_{kq}}) - |T^i_{kq}|^2 - |T^k_{iq}|^2 \}.  \label{eq:keyhalf}
\end{equation}
\end{lemma}

\begin{proof}
It follows from (\ref{eq:29}) that $sx=2s(1-t)w $. If $t\neq \frac{1}{2}$, then $x=\frac{Q}{2t-1}$, which yields that
\begin{equation}
4s(t-1)^2 w = s (3t^2-2t -s^2) \, (v_i^j - v_k^j -v_i^{\ell} + v_k^{\ell}). \label{eq:prekey}
\end{equation}
Let $i=j$ and $k=\ell$, the equation above is exactly the identity \eqref{eq:key}. When $t=\frac{1}{2}$, the linear system about $x$ and $y$ imply that $x=y$ and $Q=0$, which yields (\ref{eq:thalf}). Let $i=j$ and $k=\ell$, then (\ref{eq:keyhalf}) follows. This completes the proof of the lemma.
\end{proof}

Now we proceed to prove Theorem \ref{thm4}. This will be divided into three parts, which are the contents of Lemma \ref{lemma13}, Lemma \ref{lemma15} and Lemma \ref{lemma16} below.

\begin{lemma} \label{lemma13}
Let $(M^n,g)$ be a Hermitian manifold whose canonical metric connection $D^r_s$ is K\"ahler-like for $(r,s)\in \Omega$. Write $t=1-r+rs$. If $s\neq 0$ and $t\neq 0,1$, then $g$ is K\"ahler.
\end{lemma}

\begin{proof}
Since $s\neq 0$ and $t\neq 0,1$, by \eqref{eq:28} and \eqref{eq:30} in Lemma \ref{lemma10}, it follows that
\begin{equation}
\sum_q T^q_{ik} T^j_{\ell q} =0 \label{eq:TT}
\end{equation}
for any $i$, $j$, $k$, $\ell$. Let us denote by $A_X$ the linear transformation on $V:=T^{1,0}_pM$ associated to $X = \sum_q X_q e_q$, defined as
\[\begin{array}{crcl}
A_X:&V & \longrightarrow & V\\
&e_i & \longrightarrow & \sum\limits_{q,j} X_q T^j_{qi}e_j,\\
\end{array}\]
which is clearly independent of the choice of the unitary frame $e$. Then the equation (\ref{eq:TT}) simply means $A_XA_Y=0$ for any $X,Y\in V$. By the claim in \cite[the proof of Theorem 2]{YZ}, it yields that there exists a non-zero vector $W\in V$ such that $A_X(W) =0 $ for any $X$. Without loss of generality, we may assume that such $W$ is $e_n$. Then it follows that $T^q_{nk}=0$ for any $q$, $k$.

It is clear from (\ref{eq:key}) that $3t^2-2t-s^2= 0$ implies $T=0$, since $s\neq 0$ and $t\neq 1$. Hence, let us assume $3t^2-2t-s^2\neq 0$.
Let $i=n$ in (\ref{eq:key}) and thus it follows that that $T^n_{kq} =0$ for any $k$, $q$. This means that the direction $e_n$ plays no part in the components of the Chern torsion, namely $T^j_{ik}=0$ whenever any of the indices is $n$. Then the linear transformation $A_X$ can be restricted to the orthogonal complement $e_n^{\bot}$ of $e_n$ in $V$, and thus another direction yields, denoted by $e_{n-1}$, annihilating $A_X\big|_{e_n^{\bot}}$ for any $X$, which implies $T^{n\!-\!1}_{\, q\,k } = T^{q }_{n\!-\!1\, k}=0$ for any $q$, $k$. Repeat this argument, which yields that $T=0$, namely $g$ is K\"ahler. Therefore the proof is completed.
\end{proof}

Let us deal with the $t=1$ case, which means $r(1-s)=0$. However $s=1$ implies $r=0$ by the definition of $\Omega$. Hence, it follows that $t=1$ is actually equivalent to $r=0$, the vertical axis in $\Omega$. In other words, we are studying the K\"ahler-likeness of $D^0_s$, where $s\neq 0$. We will also assume that $s\neq  \pm 1$, as $D^0_1=\nabla$ and $D^0_{-1}=\nabla'$, and there are non-K\"ahler metric $g$ with K\"ahler-like $\nabla$ or $\nabla'$ for $n\geq3$, as mentioned in Section \ref{intro}. 



\begin{lemma} \label{lemma15}
Let $(M^n,g)$ be a Hermitian manifold whose canonical metric connection $D^0_s$ is K\"ahler-like for $s\neq 0, 1, -1$.  Then $g$ is K\"ahler.
\end{lemma}

\begin{proof}
Since $s\neq 0$ and $t=1$, the equation \eqref{eq:29} in Lemma \ref{lemma10} says that $x=0$ and thus $y=0$.
Hence, by the identity \eqref{eq:31} of Lemma \ref{lemma10}, it yields that $v^{\ell}_i = v_k^{\ell}$ since $s^2\neq 1$, namely,
\begin{equation}
 \sum_q T^{\ell }_{iq} \overline{ T^k_{jq} } = \sum_q T^{\ell }_{kq} \overline{ T^i_{jq} }  \label{eq:t1}
 \end{equation}
for any indices $i$, $j$, $k$, $\ell$.
Let $k=\ell$ and sum up in (\ref{eq:t1}) and it follows that $A_{i\overline{j}} = \overline{\phi^i_j}$ for any $i$, $j$, that is, $A=\phi^{\ast}$. Similarly, let $i=j$ and sum up in (\ref{eq:t1}), which yields $B=\phi$. Since $A$ and $B$ are Hermitian symmetric, it follows that $A=B=\phi =\phi^{\ast}$, which implies that $|T|^2=|\eta|^2$ after the trace is taken. We claim that actually it holds that $|\eta|=|T|=0$ everywhere on $M^n$.

Assume the contrary, namely, $|\eta|=|T|\neq 0$ at some point $p$ on $M^n$. This will also hold in a neighborhood of $p$. Then it enables us to modify the frame $e$ by some unitary transformation in the neighborhood of $p$, such that $\frac{\eta}{|\eta|} = \varphi_n$, yielding that $\eta_1 = \cdots =\eta_{n-1}=0$ and $\eta_n=|\eta| > 0$. It follows that $\phi^n_n=0$, which implies $A_{n\bar{n}}=B_{n\bar{n}}=0$. Therefore, $T^i_{jn}=T^n_{ij}=0$ for any $i$, $j$, which yields, under the modified frame $e$, that $\phi^j_i = T^j_{in}|\eta|=0$ for any $i$, $j$, where a contradiction $0=\mathrm{tr}(\phi)=|\eta|^2$ appears. This completes the proof.
\end{proof}

Now let us consider the $t=0$ case, which means the two branches of the hyperbola $r(1-s)=1$. The two special connections $\nabla^+$, $\nabla^-$ correspond to the case of $s=\pm2$.

\begin{lemma} \label{lemma16}
Let $(M^n,g)$ be a Hermitian manifold whose canonical metric connection $D^r_s$ is K\"ahler-like for $(r,s)\in \Omega$ such that $s\neq 0, \pm 2$ and $t=0$, where $t=1-r+rs$. Then $g$ is K\"ahler.
\end{lemma}

\begin{proof}
Since $s\neq 0$ and $t=0$, the equations \eqref{eq:27} and \eqref{eq:28} of Lemma \ref{lemma10} imply that
\begin{eqnarray}
T^j_{ik,\ell}&=&0, \label{eq:pT} \\
\sum_q ( T^q_{ik}T^j_{\ell q} + T^q_{\ell i}T^j_{k q} + T^q_{k\ell }T^j_{i q} ) & = &0. \label{eq:t0a}
\end{eqnarray}
for any $i$, $j$, $k$, $\ell$. Hence, let $j=\ell$ and sum up in \eqref{eq:t0a}, which yields
\begin{equation}\label{eq:etaT}
\sum_q \eta_q T^q_{ik}=0
\end{equation}
for any $i$, $k$. Also, the equations \eqref{eq:29} and \eqref{eq:31} of Lemma \ref{lemma10} imply that
\begin{eqnarray}
T^j_{ik,\bar{\ell}} &=& \frac{s^2}{2}(v^{\ell}_{i}-v^{\ell}_k) \label{eq:dbarT}   \\
4w &=&- s^2 (v_i^j + v_k^{\ell} - v_i^{\ell } - v_k^j ) \label{eq:t0}
\end{eqnarray}
for any $i$, $j$, $k$, $\ell$. Let $k=\ell$ and sum up in \eqref{eq:t0}, which yields
\begin{equation}
s^2 (\phi + \phi^{\ast }-B ) = (s^2-4) A. \label{eq:t0b}
\end{equation}
It follows that $|\eta|=\lambda$, where $\lambda$ is a global constant, since
\[|\eta|^2_{,\ell} = \sum_k \eta_{k,\ell}\overline{\eta_k}+\eta_k \overline{\eta_{k,\bar{\ell}}} =\frac{s^2}{2}(\sum_{i,k,q}\eta_kT_{iq}^k\overline{T^\ell_{iq}} - \eta_k \eta_q \overline{T^\ell_{kq}})=0,\]
by \eqref{eq:pT}, \eqref{eq:etaT} and \eqref{eq:dbarT}, and $|\eta|^2_{,\bar{\ell}}=0$ is similarly established.

We claim that $\lambda=0$. Assume the contrary, that is, $\lambda>0$, which enables us to choose the frame $e$ after some appropriate unitary transformation, such that $\eta=\lambda \varphi_n$, yielding that $\eta_1 = \cdots =\eta_{n-1}=0$ and $\eta_n=\lambda > 0$. Then it follows that $T^n_{ik}=0$ for any $i$, $k$, and thus $\phi^n_n=B_{n\overline{n}}=0$. Therefore, it yields from (\ref{eq:t0b}) and the assumption $s^2\neq 4$ that $A_{n\overline{n}} =0$. However, $\lambda^2 A_{n\overline{n}} = |\phi |^2$, which implies $\phi =0$ and thus $|\eta|^2=\lambda^2 =0$ after the trace is taken, where a contradiction appears. Therefore, the claim is established.

Now we have $\eta=0$, since $s\neq 0$ and $t=0$, by letting $i=j$ and $k=\ell$ and summing them up in the equation \eqref{eq:29} of Lemma \ref{lemma10}, we get $\chi = |T|^2$, hence $T=0$. This completes the proof of the lemma.
\end{proof}

\begin{proof}[\textbf{Proof of Theorem \ref{thm4}}]
Combining Lemma \ref{lemma13}, Lemma \ref{lemma15} and Lemma \ref{lemma16}, we get a proof of Theorem \ref{thm4}.
\end{proof}

The proof of  Theorem \ref{thm3} also follows.

\begin{proof}[\textbf{Proof of Theorem \ref{thm3}}]
Let us focus on the two special connections $\nabla^+$ and $\nabla^-$, namely, the case for $t=0$ and $s^2=4$. First assume that $\nabla^+$ (or equivalently $\nabla^-$) is K\"ahler-like. Lemma \ref{lemma10} for $t=0$ and $s^2=4$ leads to the following:
\begin{eqnarray}
T^j_{ik, \ell } \, &=& \, 0   \label{eq:skl1} \\
\sum_q ( T^q_{ik}T^j_{\ell q} + T^q_{\ell i}T^j_{k q} + T^q_{k\ell }T^j_{i q} )  \,&=&\,0 \label{eq:skl2} \\
T^j_{ik, \overline{\ell}} - T^{\ell }_{ik, \overline{j}} \, &=& \, 2w \label{eq:skl3} \\
T^j_{ik, \overline{\ell}} \, &=& \, 2 (v_i^{\ell } - v_k^{\ell} ) \label{eq:skl4}
\end{eqnarray}
for any $i$, $j$, $k$, $\ell$, where the index after comma means covariant derivative with respect to the Chern connection $\nabla^c=D^1=D^{1-t}$. From \eqref{eq:skl3} and \eqref{eq:skl4}, it follows that
\begin{equation}
w + v_i^j + v_k^{\ell } - v_i^{\ell } - v_k^j =0. \label{eq:skl5}
\end{equation}
Let us denote by $T^j_{ik|\ell}$ and $T^j_{ik|\overline{\ell}}$ the covariant derivatives with respect to the Strominger connection $\nabla^s=D^{-1}$. By (\ref{eq:deltaT}) and (\ref{eq:deltaTbar}), where $r=1$ and $r'=-1$, it yields that
$ T^j_{ik|\ell} = T^j_{ik,\ell}$ due to (\ref{eq:skl2}), and
$$ T^j_{ik|\overline{\ell}} - T^j_{ik,\overline{\ell}} = 2 (v_k^j - v_i^j - w).$$
Hence, $ T^j_{ik|\overline{\ell}}=0 $ due to (\ref{eq:skl4}) and  (\ref{eq:skl5}), which implies that $T$ is parallel with respect to $\nabla^s=D^{-1}$. This together with (\ref{eq:skl5}) says that $(M^n,g)$ is Strominger K\"ahler-like by \cite{ZZ}.

Conversely, if $(M^n,g)$ is Strominger K\"ahler-like, then it follows from \cite{ZZ} that (\ref{eq:skl2}), (\ref{eq:skl5}) and $T^j_{ik|\ell}= T^j_{ik|\overline{\ell}}=0 $ hold. Therefore, by (\ref{eq:deltaT}) and (\ref{eq:deltaTbar}), it yields (\ref{eq:skl1}) and (\ref{eq:skl4}), while (\ref{eq:skl3}) is a consequence of (\ref{eq:skl4}) and (\ref{eq:skl5}), which means, by definition, that $\nabla^+=D^{\!-\!1}_2$, or equivalently $\nabla^-=D^{\frac{1}{3}}_{\!-\!2}$, is K\"ahler-like. The proof of Theorem \ref{thm3} is completed.
\end{proof}

\begin{proof}[\textbf{Proof of Theorem \ref{thm5}}]
Let us assume that $D=D^r_s$ and $D'=D^{r'}_{s'}$ are both K\"ahler-like on a Hermitian manifold $(M^n,g)$, where $(r,s)\neq (r',s')\in \Omega$. Apparently, any of the following four K\"ahler-like pairs do not imply the K\"ahlerness of $g$
$$ \{ \nabla, \nabla'\},\quad \{ \nabla^+, \nabla^-\}, \quad \{ \nabla^+, \nabla^s\}, \quad \{ \nabla^-, \nabla^s\},$$
as there are examples of non-K\"ahler manifolds which are Riemannian K\"ahler-like or Strominger K\"ahler-like.

By Theorem \ref{thm2}, Theorem \ref{thm3}, and Theorem \ref{thm4}, we are left only with the case when $D=D^r$ is Gauduchon and $D'=\nabla$ is the Riemannian connection, where $r \neq \frac{1}{2}$ by Lemma \ref{lemma3}. It follows from Lemma \ref{lemma10} for $s=0$ and $t=1$ that the K\"ahler-likeness of  Riemannian connection $\nabla$ would imply
\begin{eqnarray}
T^j_{ik,\ell} &=& -\sum_q T^q_{ik}T^j_{\ell q} \label{eq:RKLpT}\\
T^j_{ik, \bar{\ell}} &=& T^{\ell}_{ik,\bar{j}} \label{eq:RKLdbarT}
\end{eqnarray}
for any $i$, $j$, $k$, $\ell$, where the index after comma stands for covariant derivative with respect to $D^{0}$. We will show that $g$ is K\"ahler in two cases: $r \neq 0$, and $r=0$.

When $r\neq0$, let us denote by $T^j_{ik|\ell}$ and $T^j_{ik|\overline{\ell}}$ the covariant derivatives with respect to the Gauduchon connection $D^{r}$. It follows from \eqref{eq:deltaT} and \eqref{eq:rTTcyclic} that $T_{ik|\ell}^j = T_{ik,\ell}^j$ for any $i$, $j$, $k$, $\ell$, and thus it yields, together with \eqref{eq:RKLpT} and \eqref{eq:rTderi}, that
\begin{eqnarray}
T_{ik|\ell}^j\ \ =\ \ T_{ik,\ell}^j &=& 0  \label{eq:RrpT} \\
\sum_q T_{ik}^q T^j_{\ell q} &=& 0  \label{eq:RrTT}
\end{eqnarray}
for any $i$, $j$, $k$, $\ell$. Similarly, the identity \eqref{eq:deltaTbar} implies that
\begin{equation}\label{eq:RrdbarT}
T^j_{ik| \bar{\ell}}= T^j_{ik ,\bar{\ell}} -r \{ T^j_{qk} \overline{T^i_{q\ell }} + T^j_{iq} \overline{T^k_{q\ell }}
- T^q_{ik} \overline{T^q_{j\ell }}\}
\end{equation}
for any $i$, $j$, $k$, $\ell$, which yields, together with \eqref{eq:RKLdbarT} and Lemma \ref{lemma5}, that
\[\frac{(r-1)(3r-1)}{2(2r-1)} |\eta|^2 - \frac{(r-1)^2}{2(2r-1)} |T|^2 = r |\eta|^2,\] or equivalently
\begin{equation}\label{eq:Rrchi}
(r-(\sqrt{2}-1))(r+\sqrt{2}+1) |\eta|^2 + (r-1)^2 |T|^2 = 0.
\end{equation}
Therefore, it is apparent that $r \leq -\sqrt{2}-1$ or $r \geq \sqrt{2}-1$, $r \neq 1$ implies $|T|=0$ and thus $g$ is K\"ahler.

Let us deal with the remaining cases under the condition $r \neq 0$, namely $ -\sqrt{2}-1 < r < \sqrt{2}-1$ or $ r =1$.
If $ -\sqrt{2}-1 < r < \sqrt{2}-1$ and $r\neq 0$, then by \eqref{eq:Rrchi} we have
\[|T|^2 = - \frac{(r-(\sqrt{2}-1))(r+\sqrt{2}+1)}{(r-1)^2} |\eta|^2.\]
Multiply $\bar{\eta}_{k}\bar{\eta}_{\ell}$ on both sides of \eqref{eq:RrTT} and sum $k,\ell$ up, which yields that
\[\sum_{q} \phi_i^q \phi_q^j =0\]
for any $i$, $j$. The nilpotency of the matrix $\phi$ implies that $\mathrm{tr}(\phi)=|\eta|^2=0$ and thus $|T|=0$. On the other hand, if $r=1$, it follows from \eqref{eq:Rrchi} that $|\eta|=0$. Note that $T^j_{ik|\bar{\ell}}=0$ in this case from Lemma \ref{lemma3}.
Together with \eqref{eq:RKLdbarT}, the identity \eqref{eq:RrdbarT} implies that
\[\sum_q \{T^j_{qk}\overline{T^i_{q \ell}} - T^{\ell}_{qk}\overline{T^i_{q j}}
- T^j_{qi}\overline{T^k_{q \ell}} + T^{\ell}_{q i} \overline{T^k_{q k}} - 2 T^q_{ik}\overline{T^k_{j \ell}} \} =0.\]
Let $i=j$, $k=\ell$ in the equation above and sum $i$ up, which yields that
\begin{equation} \label{eq:Trc}
\sum_{i,q}|T^k_{q i}|^2 = \sum_{i,q} |T^q_{ik}|^2
\end{equation}
for any $k$. By the same trick applied in the proof of Lemma \ref{lemma13},
\eqref{eq:RrTT} enables us to assume that $e_n$ annihilates the transformation $A_X$, yielding that $T_{nq}^i=0$ for any $q$, $i$.
Hence, it follows that $T^n_{qi}=0$ for any $q$, $i$ from \eqref{eq:Trc}, which implies that the Chern torsion $T$ has vanishing components whenever one of the indices is $n$. As in the proof of Lemma \ref{lemma13}, the trick can be repeated to deduce $T=0$ in the end. Therefore, the proof for the case when $r \neq 0$ is completed.

When $r=0$, Lemma \ref{lemma5} and \eqref{eq:RKLdbarT} imply $|\eta|^2=|T|^2$. From Proposition \ref{GKLr=0}, the K\"ahler-likeness of the Lichnerowicz connection $D^0$ implies that $\eta=0$ everywhere, yielding that $|T|=0$. In summary, the proof of Theorem \ref{thm5} is completed.
\end{proof}


\noindent\textbf{Acknowledgments.} We would like to thank mathematicians Jixiang Fu, Gabriel Khan, Kefeng Liu, Luigi Vezzoni, Bo Yang, Xiaokui Yang, Shing-Tung Yau and Xianchao Zhou for their interests and/or help. We are also indebted to the work \cite{AOUV} and \cite{Fu-Zhou} which inspired this study.

\vs

\end{document}